\documentclass{elsarticle}

\usepackage{graphicx}
\usepackage[hidelinks]{hyperref}

\usepackage{amsfonts}
\usepackage{epstopdf}
\usepackage{algorithmic}
\usepackage{amssymb,amsmath,epsfig,multirow}
\usepackage[displaymath, mathlines,running]{lineno}
\usepackage{mathtools} 
\usepackage{color}
\usepackage{lineno,hyperref}
\usepackage{rotating}
\usepackage{makeidx}
\usepackage{float}
\usepackage{epsfig}
\usepackage{amsmath,amssymb}
\usepackage{enumitem}
\usepackage{epic,eepic}
\usepackage{overpic}
\usepackage{color}
\usepackage{amsfonts}
\usepackage{graphicx}
\usepackage{enumitem}
\usepackage{ulem}

\newtheorem{theorem}{Theorem}[section]
\newtheorem{lemma}[theorem]{Lemma}
\newtheorem{proposition}[theorem]{Proposition}
\newtheorem{corollary}[theorem]{Corollary}

\newproof{proof}{Proof}

\begin{document}
	
\thispagestyle{empty}

\noindent	The version of record of this article, first published in Applied Numerical Mathematics, is available online at Publisher’s website:

\noindent \url{https://doi.org/10.1016/j.apnum.2022.06.012}

\newpage
\setcounter{page}{1}

\begin{frontmatter}

\title{A note on some bounds between cubic spline interpolants depending on the boundary conditions:
	Application to a monotonicity property}
\tnotetext[label1]{Funding: This research has been supported by grant PID2020-117211GB-I00 funded by MCIN/AEI/10.13039/501100011033.}

\author[UV1]{Antonio Baeza\corref{cor1}}
\ead{antonio.baeza@uv.es}
\author[UV1]{Dionisio F. Y\'a\~nez}
\ead{dionisio.yanez@uv.es}
\date{Received: date / Accepted: date}
\address[UV1]{Departament de Matem\`atiques,  Universitat de Val\`encia, 46100, Burjassot (Spain)}
\cortext[cor1]{Corresponding author.}
\begin{abstract}
In the context of cubic splines, the authors have contributed to a recent paper dealing with the
computation of nonlinear derivatives at the interior nodes so that monotonicity is enforced while keeping the order of
 approximation of the spline as high as possible.
	During the review process of that paper, one of the reviewers raised the question of whether a cubic spline
	interpolating monotone data could be forced to preserve monotonicity by imposing suitable values of the first
	derivative at the endpoints.
	  	Albeit a negative answer appears to be intuitive, we have found no results regarding this fact. In this short
	  	work we prove that the answer to that question is actually negative.
 \end{abstract}


\end{frontmatter}

%%-----------------------------
%%      your text
%%-----------------------------
\section{Introduction}
Approximation techniques are used in applications as design of curves, surfaces, robotics, creation of pieces in industry

In many applications that use approximations of given data by means of smooth curves or surfaces, additional properties like monotonicity and convexity preservation are often required. In particular, interpolation by cubic splines is a
well-known technique that provides twice continuously differentiable curves. The problem consists on constructing a piecewise polynomial with equal first and second derivative values at the nodes, so that the interpolant is globally twice continuously differentiable.

In order to solve the problem, values for the first or second derivative at the boundary points are necessary. For natural splines, the values for the second derivative at the endpoints are set to zero, while for interpolating splines first derivative values are forced to coincide with those of the function derivative at the endpoint.

Splines do not preserve data monotonicity and consequently many works have dealt with the problem of monotonicity-preserving spline interpolation, being \cite{FrCar} a seminal paper on the matter.

The authors have participated in a contribution \cite{Arandigabaezayanez}  where a monotonicity-preserving reconstruction is proposed using nonlinear reconstructions of the
first derivatives at the interior points. During the review process of the paper one of the reviewers raised the following question:
{\it By changing the first derivative at the endpoints of the interval, is it possible to find twice continuously differentiable monotonicity preserving cubic spline interpolants?}

In this note we show that the answer to this question is negative. We first prove a bound on the distance between
 two splines that interpolate the same data but differ on the values imposed at the boundary.
   We then show that monotonicity preservation is not possible by simply changing these boundary data.
   The paper is organized as follows: in Section \ref{sec:notation} we introduce the notation and basic facts about cubic splines. In Section
   \ref{sec:main}
   we prove the main result
   and in Section \ref{sec:monotonicity} the main result of the paper stating that finding a monotonicity-preserving cubic spline by manipulating the derivative values at the interval endpoints is not always possible.

\section{Notation. Cubic splines in Hermite form}\label{sec:notation}

We use the same notation as in \cite{Arandigabaezayanez}. We start with the problem of piecewise cubic Hermite interpolation.  Let $x_{1}< x_{2} < \hdots< x_{n}$ be a partition of the interval
$[x_{1} ,x_{n}]$ and let $f_{i} = f(x_{i})$ be the values of a
certain function at the knots {$\{x_{i}\}_{i=1}^n$}. Given approximate values of the first derivative of $f$ at the knots{,} denoted by $\{\dot{f}_{i}\}_{i=1}^n,$ the problem is to construct a  piecewise cubic polynomial function $P(x),$ composed by $n-1$
cubic polynomial pieces $P_i(x)$ respectively defined on the intervals $[x_i,x_{i+1}],$ that satisfy
$P(x_i)=f_i$,  $P(x_{i+1})=f_{i+1}$, $P{'}(x_i)=\dot{f}_i$,  $P{'}(x_{i+1})=\dot{f}_{i+1}$, for $1\leq i < n$. If we denote as $\Delta f_{i} = f_{i+1} - f_{i}$ and
$m_{i} ={\Delta f_{i}}/{h_{i}}$, where
 $h_{i} = x_{i+1} -x_{i}$ are the mesh spacings and $\hat{h}=\max_{i=1,\hdots,n-1} (h_i)$, then, the $i$-th polynomial
$P_i(x),  x \in [x_{i},x_{i+1}]$ (see \cite{Boo01}
for details),
has the form
\begin{equation}\label{Hermite}
P_{i}(x) = c^i_{1} + c^i_{2} (x - x_{i}) + c^i_{3} (x-x_{i})^{2} +
c^i_{4}(x-x_{i})^{2}(x-x_{i+1}),
\end{equation}
where  $c^i_{1} := f_{i}$, $c^i_{2} := f[x_{i},x_{i}] = \dot{f}_{i}$,
 $c^i_{3} := f[x_{i},x_{i}, x_{i+1}]
= ( {m_{i} -\dot{f}_{i}} )/ {h_{i}}$, 
 $c^i_{4} := f[x_{i},x_{i},x_{i+1},x_{i+1}] =
( {\dot{f}_{i+1} + \dot{f}_{i} - 2 m_{i}} ) / {h_{i}^{2}}$.

Cubic splines are constructed by computing approximate derivative values that enforce continuity of the second derivative, by solving the following problem:
\begin{eqnarray}\label{splinesconditions1}
& &P_i''(x_{i+1})=P_{i+1}''(x_{i+1}),\\\label{splinesconditions2}
& &P_1'(x_1)=\dot{f}_1, \quad P_{n-1}'(x_n)=\dot{f}_n,
\end{eqnarray}
with $i=1,\hdots,n-2$, where $\dot{f}_1$ and $\dot{f}_n$ are given boundary conditions at the interval endpoints.

It is easy to get that {\eqref{splinesconditions1}--\eqref{splinesconditions2} corresponds to the following linear system} (see more details in \cite{Arandigabaezayanez}):
\begin{equation}
\label{eq:splind}
\left\{
\begin{array}{lc}
2\, \dot{f}_2+ \mu_1\, \dot{f}_3=3\, (\lambda_1 \, m_{1} +\mu_1 \, m_{2})- \lambda_1 \,\dot{f}_{1}=:b_1, & \\
\lambda_i \, \dot{f}_{i}+ 2 \, \dot{f}_{i+1}+ \mu_i \, \dot{f}_{i+2}=3\,(\lambda_{i} \, m_{i} +\mu_{i} \, m_{i+1})=:b_i,& i=2,\ldots,n-3, \\
\lambda_{n-2} \, \dot{f}_{n-2}+ 2 \, \dot{f}_{n-1}=3\, (\lambda_{n-2} \, m_{n-2} +\mu_{n-2} \, m_{n-1})-  \mu_{n-2} \, \dot{f}_{n}=:b_{n-2},&
  \\
\end{array}
\right.
\end{equation}
where $\lambda_i=\frac{h_{i+1}}{h_i+h_{i+1}}$,  $\mu_i=\frac{h_i}{h_i+h_{i+1}}$, $\lambda_i+\mu_i=1$, $i=1,\hdots,n-2$.

We can write the system in matrix form as
$A \dot F = B,$
with $\dot F=\left[\dot{f}_2,\dot{f}_3,\hdots,\dot{f}_{n-1}\right]^T$ and
\begin{equation}\label{splinesystem}
A=\left[\begin{array}{cccccc}
2&\mu_1&&&&\\
\lambda_2&2&\mu_2&&&\\
%&&&&&&&&&&&\\
&\lambda_3&2&\mu_3&\\
&&\ddots&\ddots &\ddots&\\
%&&&&&&&&&&&\\
%&&&&&&&&&&&\\
&&&\lambda_{n-3}&2&\mu_{n-3}\\
&&&&\lambda_{n-2}&2\\
\end{array}
\right], B=\left[\begin{array}{c}
 3(\lambda_1 \, m_{1} +\mu_1 \, m_{2})-\lambda_1 \,\dot{f}_1\\
 3(\lambda_2 \, m_{2} +\mu_2 \, m_{3})\\
 3(\lambda_3 \, m_{3} +\mu_3 \, m_{4}) \\
%\vdots\\d_{u-1}\\d_u \\d_{u+1} \\
%\vdots\\d_{v-1}\\d_v \\d_{v+1} \\
\vdots
\\
 3(\lambda_{n-4} \, m_{n-4} +\mu_{n-4} \, m_{n-3})\\
 3(\lambda_{n-3} \, m_{n-3} +\mu_{n-3} \, m_{n-2})\\
 3(\lambda_{n-2} \, m_{n-2} +\mu_{n-2} \, m_{n-1})-\mu_{n-2} \dot{f}_n\\
\end{array}
\right]=\left[\begin{array}{c}
b_1\\
b_2 \\
b_3 \\
\vdots\\
%M_{u-1}\\M_u \\M_{u+1} \\\vdots\\
%M_{v-1}\\M_v \\M_{v+1} \\\vdots\\
b_{n-4}\\
b_{n-1}\\
b_{n-2}
\end{array}
\right].
\end{equation}
The system matrix is strictly diagonally dominant and hence non-singular. If we denote by $A^{-1}_{i,j}$, $1\leq i,j\leq n-2$ the {entries} of the matrix $A^{-1}$  we have
\begin{equation}
\dot f_i=\sum_{j=1}^{n-2}A^{-1}_{i-1,j}b_{j}, \quad 2\leq i\leq n-1.
\end{equation}
The structure of the entries of the matrix $A^{-1}$ is crucial to prove that the images of two splines with different conditions \eqref{splinesconditions2} become more and more similar as the points get far from the boundary. We state a result (a particular case of the result in \cite{kershaw1}) that provides a bound on the values $A^{-1}_{i,j}$.

\begin{lemma}\label{lemma2}
Let $0\leq \mu_i,\lambda_i\leq 1$ with $1\leq i\leq n-2$ be such that
$$\lambda_i+\mu_i=1\quad i=1,\hdots,n-2,$$
and let $A\in \mathbb{R}^{(n-2)\times (n-2)}$ be the matrix defined
 by \eqref{splinesystem}. Then $$0< (-1)^{i-j}A_{ij}^{-1}\leq \frac{2}{3}\cdot2^{-|i-j|},\quad 1\leq i,j\leq n-2. $$
\end{lemma}

%%%%%%%%%%%%%
\section{Distance between two spline interpolations with different boundary conditions}\label{sec:main}
%%%%%%%%%%%%%%%

In this section we apply Lemma \ref{lemma2} to compute a bound on the distance between two splines constructed using different boundary values.

\begin{proposition}\label{prop1}
Let $P(x)$ and $\widetilde P(x)$ be the cubic spline interpolants of the form \eqref{Hermite} for the data $\{x_i, f(x_i), \dot{f}_i\}$ and  $\{x_i, f(x_i), \dot{\widetilde{f}}_i\}_{i=1}^n$ respectively, with $\dot{\widetilde{f}}_i = \dot{f}_i$ for $1< i < n$ (computed by solving system \eqref{splinesystem})
and $\dot{\widetilde{f}}_i \neq \dot{f}_i$ for $i=1, n$.  For $2\leq i\leq n-2$, $P_i(x)$ and $\widetilde P_i(x)$ denote the corresponding polynomial pieces
for each spline. Then the following inequality holds:
\begin{equation}\label{final}
|P_i(x)-\widetilde P_i(x)|\leq 8h_i(2^{-i}\lambda_1|\dot{\widetilde{f}}_1-\dot{f}_1|+\mu_{n-2}2^{i-n}|\dot{\widetilde{f}}_n-\dot{f}_n|), \,\,\forall \,x\in[x_i,x_{i+1}].
\end{equation}
\end{proposition}
\begin{proof}
Let us start calculating
\begin{equation}
b_1-\widetilde{b}_1= \lambda_1(\dot{\widetilde{f}}_1-\dot{f}_1),\quad b_i-\widetilde{b}_i=0,\,\, 2\leq i \leq n-1,\quad b_{n-2}-\widetilde{b}_{n-2}= \mu_{n-2}(\dot{\widetilde{f}}_n-\dot{f}_n).
\end{equation}
Then, for $2\leq i\leq n-1$ it holds
\begin{equation}\label{eq:disfp}
\dot f_i-\dot{\widetilde{f}}_i=\sum_{j=1}^{n-2}A^{-1}_{i-1,j}b_{j}-\sum_{j=1}^{n-2}A^{-1}_{i-1,j}\widetilde{b}_{j}=\sum_{j=1}^{n-2}A^{-1}_{i-1,j}(b_j-\widetilde{b}_j)=
A^{-1}_{i-1,1}\lambda_1(\dot{\widetilde{f}}_1-\dot{f}_1) + A^{-1}_{i-1,n-2}\mu_{n-2}(\dot{\widetilde{f}}_n-\dot{f}_n).
\end{equation}
\begin{equation*}
(\dot f_{i+1}-\dot{\widetilde{f}}_{i+1})+(\dot f_i-\dot{\widetilde{f}}_i)=
(A^{-1}_{i-1,1}+A^{-1}_{i,1})\lambda_1(\dot{\widetilde{f}}_1-\dot{f}_1) + (A^{-1}_{i-1,n-2}+A^{-1}_{i,n-2})\mu_{n-2}(\dot{\widetilde{f}}_n-\dot{f}_n), \quad 2\leq i\leq n-2.
\end{equation*}
The distance between $P_i(x)$ and $\widetilde P_i(x)$ for all $x\in[x_i,x_{i+1}]$, $i=2,\hdots, n-2$, can be obtained with:
\begin{equation*}
\begin{split}
P_i(x)-\widetilde P_i(x)=&(\dot f_i-\dot{\widetilde{f}}_i)(x-x_i)-(\dot f_i-\dot{\widetilde{f}}_i)\frac{(x-x_i)^2}{h_i}+\left((\dot f_{i+1}-\dot{\widetilde{f}}_{i+1})+(\dot f_i-\dot{\widetilde{f}}_i)\right)\frac{(x-x_i)^2(x-x_{i+1})}{h^2_i}\\
=&\lambda_1(\dot{\widetilde{f}}_1-\dot{f}_1)(x-x_i)\left(A^{-1}_{i-1,1}+A^{-1}_{i-1,1}\frac{(x-x_i)}{h_i}+(A^{-1}_{i-1,1}+A^{-1}_{i,1})\frac{(x-x_i)(x-x_{i+1})}{h^2_i} \right)\\
+&\mu_{n-2}(\dot{\widetilde{f}}_n-\dot{f}_n)(x-x_i)\left(A^{-1}_{i-1,n-2}+A^{-1}_{i-1,n-2}\frac{(x-x_i)}{h_i}+(A^{-1}_{i-1,n-2}+A^{-1}_{i,n-2})\frac{(x-x_i)(x-x_{i+1})}{h^2_i} \right).
\end{split}
\end{equation*}
From $|A^{-1}_{i-1,1}+A^{-1}_{i,1}|\leq |A^{-1}_{i-1,1}|$, $|A^{-1}_{i-1,n-2}+A^{-1}_{i,n-2}|\leq |A^{-1}_{i,n-2}|$ and $|A^{-1}_{i-1,n-2}|\leq |A^{-1}_{i,n-2}|$,  $i=2,\hdots, n-2$ by Lemma \ref{lemma2}, we get:
\begin{equation*}
\begin{split}
|P_i(x)-\widetilde P_i(x)|\leq &\lambda_1|A^{-1}_{i-1,1}||\dot{\widetilde{f}}_1-\dot{f}_1|(x-x_i)\left(1+\frac{(x-x_i)}{h_i}+\frac{(x-x_i)(x_{i+1}-x)}{h^2_i} \right)\\
+&\mu_{n-2}|A^{-1}_{i,n-2}||\dot{\widetilde{f}}_n-\dot{f}_n|(x-x_i)\left(1+\frac{(x-x_i)}{h_i}+\frac{(x-x_i)(x_{i+1}-x)}{h^2_i} \right)\\
=&(\lambda_1|A^{-1}_{i-1,1}||\dot{\widetilde{f}}_1-\dot{f}_1|+\mu_{n-2}|A^{-1}_{i,n-2}||\dot{\widetilde{f}}_n-\dot{f}_n|)\left((x-x_i)+\frac{(x-x_i)^2}{h_i}+\frac{(x-x_i)^2(x_{i+1}-x)}{h^2_i} \right).
\end{split}
\end{equation*}
Finally, by Lemma \ref{lemma2} it holds $|A^{-1}_{i-1,1}|\leq \frac{8}{3}2^{-i}$ and $|A^{-1}_{i,n-2}|\leq \frac{8}{3}2^{i-n}$. Then:
\begin{equation}
\begin{split}
|P_i(x)-\widetilde P_i(x)|\leq &8h_i(2^{-i}\lambda_1|\dot{\widetilde{f}}_1-\dot{f}_1|+\mu_{n-2}2^{i-n}|\dot{\widetilde{f}}_n-\dot{f}_n|).
\end{split}
\end{equation}
\end{proof}

As a consequence of Proposition \ref{prop1} we can prove the following result regarding the order of approximation of the spline
 (known result, showed e.g. in \cite{kershaw3}).
\begin{corollary}
Assume that $\hat h < 1$, $f(x)\in C^{4}([x_1,x_n])$ and  let $L>0$ be such that  $|f^{(4)}(x)|\leq L$, for all $x\in[x_1,x_n]$. Let $P(x)$ be the interpolating spline of
$f$ and $\widetilde{P}$ be another cubic spline constructed with the same data as $P(x)$ except for the first derivative values at the interval endpoints that can take any value.  If there exist $K,p>0$, such that
$\hat{h}/ h_j \leq K$ for all $j=1,\hdots,n-1$  and $\Omega:=\{i\in \mathbb{N}:-p\log_2(\hat{h})<i<n+p\log_2(\hat{h})\}\neq \emptyset$ then
$$\max_{x\in[x_i,x_{i+1}],i\in\Omega}|\widetilde{P}_i(x)-f(x)|=  O(\hat{h}^{\min\{p+1,4\}}). $$
\end{corollary}
\begin{proof}
If $P$ is the spline interpolator constructed with the solution of system $A \dot F = B$ defined in \eqref{splinesystem} with clamped boundary conditions $\dot{f}_1=f'(x_1)$, $\dot{f}_n=f'(x_n)$ then it is known \cite{Stoerbulirsch} that
$$\max_{x\in[x_i,x_{i+1}]}|{P}_i(x)-f(x)|=  O(\hat{h}^4).$$
On the other hand, $i\in \Omega$ implies $2^{-i}< \hat{h}^p$ and $2^{n-i}< \hat{h}^p$. The result hence follows from  Proposition \ref{prop1}.
\end{proof}

\section{Main result}\label{sec:monotonicity}

In this section we show that given a cubic spline that does not preserve data monotonicity, is it not always possible to obtain another cubic spline that does by only changing the first derivative at the endpoints. We recall following theorem \cite{FrCar}:

\begin{theorem}(Necessary conditions {for monotonicity}.)
\label{FrCarl}
Let $P_{i}$ be a monotone cubic Hermite interpolant
of the data $\{(x_{i},f_{i}, \dot f_i), (x_{i+1},f_{i+1},\dot f_{i+1})\}$. Then:
\begin{equation}\label{signos}
sign(\dot{f}_{i}) = sign(\dot{f}_{i+1}) = sign(m_{i}).
\end{equation}
Furthermore, if $m_{i} = 0$ then $P_{i}$ is monotone
(constant)
if and only if  $\dot{f}_{i} = \dot{f}_{i+1} = 0 $.
\end{theorem}

We prove the main proposition of the paper.

\begin{proposition}\label{prop:1}
Let $n$ be an odd number and $\{(x_{i},f_i)\}_{i=1}^n$, with $f_1\leq f_2\leq \hdots\leq f_n$  be the data used to construct a spline $P$ by solving the system given in Eq. \eqref{splinesystem} with given boundary data $\dot{f}_1$, $\dot{f}_n$ which satisfies:
\begin{itemize}
\item $P$ is monotone in $[x_1,x_2]$ and $[x_{n-1},x_n]$.
\item
There exists $i_0\in \mathbb{N}$ with $\max(1, 1-\log_2(R_1))\leq i_0 <\min(n-2, n-3+\log_2(R_n))$ such that
$
P_{i_0}(x_{i_0+\frac{1}{2}})-f_{i_0+1}>0,
$ being $x_{i_0+\frac{1}{2}}=x_{i_0}+\frac{h_{i_0}}{2}$ ,
$R_1=\frac{P_{i_0}(x_{i_0+\frac{1}{2}})-f_{i_0+1}}{8K_1\lambda_1}$
with
\begin{equation*}
K_1= 3 (A^{-1}_{1,1}(\lambda_1 \, m_{1} +\mu_1 \, m_{2}) + A^{-1}_{1,n-2}(\lambda_{n-2} \, m_{n-2} +\mu_{n-2} \, m_{n-1})) + \sum_{j=2}^{n-3}A^{-1}_{1,j}b_{j}+4m_1,
\end{equation*}
and
$R_n=\frac{P_{i_0}(x_{i_0+\frac{1}{2}})-f_{i_0+1}}{8\mu_{n-2}K_n}$
with
\begin{equation*}
K_n= 3 (A^{-1}_{n-2,1}(\lambda_1 \, m_{1} +\mu_1 \, m_{2}) + A^{-1}_{n-2,n-2}(\lambda_{n-2} \, m_{n-2} +\mu_{n-2} \, m_{n-1})) + \sum_{j=2}^{n-3}A^{-1}_{n-2,j}b_{j}+4m_{n-1}.
\end{equation*}
\end{itemize}

If a new spline $\widetilde{P}$ is constructed with the
same data as $P$ but changing the values $\dot{f}_1$ and $\dot{f}_n$
by $\dot{\widetilde{f}}_1$ and $\dot{\widetilde{f}}_n$ then $\widetilde{P}$
is not monotone in $[x_1,x_2]$ or $[x_{n-1},x_n]$ or
$\left(\widetilde{P}_{i_0}(x_{i_0+\frac{1}{2}})-f_{i_0+1}\right)>0$,
i.e., the new spline is not monotone.

\end{proposition}
\begin{proof}
As $f_1\leq \hdots \leq f_n$, in order to obtain a monotone spline we have that any
approximation to the first derivative at any node has to be higher or equal to 0 by
 Theorem \ref{FrCarl}. By Lemma \ref{lemma2} it holds $A^{-1}_{1,1},A^{-1}_{1,n-2}>0$ since $n$
  is an odd number. Thus, we get
\begin{equation*}
\begin{split}
\dot{f}_2=&\sum_{j=1}^{n-2}A^{-1}_{1,j}b_{j}=A^{-1}_{1,1}b_1 +A^{-1}_{1,n-2}b_{n-2}+\sum_{j=2}^{n-3}A^{-1}_{1,j}b_{j} \\
=& A^{-1}_{1,1}(3\, (\lambda_1 \, m_{1} +\mu_1 \, m_{2})- \lambda_1 \,\dot{f}_{1})+ A^{-1}_{1,n-2}(3\, (\lambda_{n-2} \, m_{n-2} +\mu_{n-2} \, m_{n-1})
-  \mu_{n-2} \, \dot{f}_{n}) + \sum_{j=2}^{n-3}A^{-1}_{1,j}b_{j}\\
\leq& 3 (A^{-1}_{1,1}(\lambda_1 \, m_{1} +\mu_1 \, m_{2}) + A^{-1}_{1,n-2}(\lambda_{n-2} \, m_{n-2} +\mu_{n-2} \, m_{n-1})) + \sum_{j=2}^{n-3}A^{-1}_{1,j}b_{j}=:k_1
\end{split}
\end{equation*}
As $P_1$ is monotone in $[x_1,x_2]$ then
\begin{equation*}
\begin{split}
& f_1\leq P_1(x_{1+\frac{1}{2}})=\frac{f_1+f_2}{2}+\frac{h_1}{8}\left(\dot{f}_1-\dot{f}_2\right)\leq f_2 \Rightarrow  \dot{f}_2+4\left(\frac{f_1-f_2}{h_1}\right) \leq \dot{f}_1 \leq \dot{f}_2+4\left(\frac{f_2-f_1}{h_1}\right) \Rightarrow \\
& 0 \leq \dot{f}_1 \leq \dot{f}_2+4m_1\leq k_1+4m_1=:K_1
\end{split}
\end{equation*}
Analogously:
\begin{equation*}
\begin{split}
\dot{f}_{n-1}=&\sum_{j=1}^{n-2}A^{-1}_{n-2,j}b_{j}=A^{-1}_{n-2,1}b_1 +A^{-1}_{n-2,n-2}b_{n-2}+\sum_{j=2}^{n-3}A^{-1}_{n-2,j}b_{j} \\
=& A^{-1}_{n-2,1}(3\, (\lambda_1 \, m_{1} +\mu_1 \, m_{2})- \lambda_1 \,\dot{f}_{1})+ A^{-1}_{n-2,n-2}(3\, (\lambda_{n-2} \, m_{n-2} +\mu_{n-2} \, m_{n-1})
-  \mu_{n-2} \, \dot{f}_{n}) + \sum_{j=2}^{n-3}A^{-1}_{n-2,j}b_{j}\\
\leq& 3 (A^{-1}_{n-2,1}(\lambda_1 \, m_{1} +\mu_1 \, m_{2}) + A^{-1}_{n-2,n-2}(\lambda_{n-2} \, m_{n-2} +\mu_{n-2} \, m_{n-1})) + \sum_{j=2}^{n-3}A^{-1}_{n-2,j}b_{j}=:k_{n},
\end{split}
\end{equation*}
as $P_{n-1}$ is monotone in $[x_{n-1},x_n]$ then
\begin{equation*}
\begin{split}
f_{n-1}&\leq P_{n-1}(x_{n-\frac{1}{2}})=\frac{f_{n-1}+f_{n}}{2}+\frac{h_{n-1}}{8}\left(\dot{f}_{n-1}-\dot{f}_n\right)\leq f_n \\
&\Rightarrow  \dot{f}_{n-1}+4\left(\frac{f_{n-1}-f_{n}}{h_{n-1}}\right) \leq \dot{f}_n \leq \dot{f}_{n-1}+4\left(\frac{f_n-f_{n-1}}{h_{n-1}}\right) \\
&\Rightarrow 
 0 \leq \dot{f}_n \leq \dot{f}_{n-1}+4m_{n-1}\leq k_n+4m_{n-1}=:K_n.
\end{split}
\end{equation*}
The same argument applies to any spline $\widetilde{P}$ that is monotone in $[x_1,x_2]$ and $[x_{n-1},x_n]$ and then the same conditions have to be satisfied by $\widetilde{P}$  i.e.
$0 \leq \dot{\widetilde{f}}_1 \leq K_1$, $0 \leq \dot{\widetilde{f}}_n \leq K_n$. By hypothesis we know that
$$\frac{f_{i_0}+f_{i_0+1}}{2}+\frac{h_{i_0}}{8}\left(\dot{f}_{i_0}-\dot{f}_{i_0+1}\right)=P_{i_0}(x_{i_0+\frac{1}{2}})>f_{i_0+1}.$$
We take
	$$\varepsilon:=P_{i_0}(x_{i_0+\frac{1}{2}})-f_{i_0+1}= \frac{f_{i_0}-f_{i_0+1}}{2}+\frac{h_{i_0}}{8}\left(\dot{f}_{i_0}-\dot{f}_{i_0+1}\right)>0$$
	and by Lemma \ref{lemma2}, from $1-\log_2(R_1)\leq i_0 <n-2$ and $n-3\geq i_0-\log_2(R_n)$ we obtain:
\begin{equation*}
|A_{i_0,1}^{-1}|\leq \frac{2}{3}\cdot2^{-|i_0-1|}\leq \frac{\varepsilon}{12\lambda_1K_1},\quad |A_{i_0+1,n-2}^{-1}|\leq \frac{2}{3}\cdot2^{-|n-3-i_0|}\leq \frac{\varepsilon}{12\mu_{n-2}K_n}.
\end{equation*}
If we denote as $\dot{\widetilde{f}}_{i_0}$, $\dot{\widetilde{f}}_{i_0+1}$ the new values obtained by changing the first derivative values $\dot{f}_1$ and $\dot{f}_n$ by $\dot{\widetilde{f}}_{1}$ and $\dot{\widetilde{f}}_{n}$, then, according to \eqref{eq:disfp} the distance with the original derivative approximation is:
\begin{equation*}
\dot{\widetilde{f}}_{i_0}-\dot{f}_{i_0}=A^{-1}_{i_0,1}\lambda_1({f}'_{1}-\dot{\widetilde{f}}_{1})+A^{-1}_{i_0,n-2}\mu_{n-2}({f}'_{n}-\dot{\widetilde{f}}_{n});\quad \dot{\widetilde{f}}_{i_0+1}-\dot{f}_{i_0+1}=A^{-1}_{i_0+1,1}\lambda_1({f}'_{1}-\dot{\widetilde{f}}_{1})+A^{-1}_{i_0+1,n-2}\mu_{n-2}({f}'_{n}-\dot{\widetilde{f}}_{n}),
\end{equation*}
Finally,
\begin{equation*}
\begin{split}
\widetilde{P}_{i_0}(x_{i_0+\frac{1}{2}})-f_{i_0+1}=&\frac{f_{i_0}-f_{i_0+1}}{2}+\frac{h_{i_0}}{8}\left(\dot{\widetilde{f}}_{i_0}-\dot{\widetilde{f}}_{i_0+1}\right)
= \frac{f_{i_0}-f_{i_0+1}}{2}+\frac{h_{i_0}}{8}\left(\dot{f}_{i_0}-\dot{f}_{i_0+1}\right) \\&+( A^{-1}_{i_0,1}-A^{-1}_{i_0+1,1})\lambda_1({f}'_{1}-\dot{\widetilde{f}}_{1})+(A^{-1}_{i_0,n-2}-A^{-1}_{i_0+1,n-2})\mu_{n-2}({f}'_{n}-\dot{\widetilde{f}}_{n})\\
&=\varepsilon+( A^{-1}_{i_0,1}-A^{-1}_{i_0+1,1})\lambda_1({f}'_{1}-\dot{\widetilde{f}}_{1})+(A^{-1}_{i_0,n-2}-A^{-1}_{i_0+1,n-2})\mu_{n-2}({f}'_{n}-\dot{\widetilde{f}}_{n})\\
&\geq\varepsilon-4|A^{-1}_{i_0,1}|\lambda_1K_1-4|A^{-1}_{i_0+1,n-2}|\mu_{n-2}K_n\\
&\geq\varepsilon-\frac{4\lambda_1K_1}{12\lambda_1K_1}\varepsilon-\frac{4\mu_{n-2}K_n}{12\mu_{n-2}K_n}\varepsilon=\varepsilon-\frac{2\varepsilon}{3}=\frac{\varepsilon}{3}>0.
\end{split}
\end{equation*}
\end{proof}

As a conclusion we remark that the form of the
matrix $A$ in \eqref{splinesystem} and its inverse
lead to splines that are close to each other
in intervals far from the boundary when different values for the first derivatives are set at the endpoints.

{\bf Acknowlegments:} We thank the reviewer of \cite{Arandigabaezayanez} who raised the question that inspired this paper.

\bibliography{yanez3}

\end{document}